\documentclass[12pt]{amsart}
\usepackage{latexsym, url}
\usepackage{amsthm}
\usepackage{amsmath}
\usepackage{amsfonts}
\usepackage{amssymb}
\usepackage{showkeys}
\usepackage[dvips]{graphicx}
\usepackage{xypic}
\input xy
\xyoption{all}

\newtheorem{theo}{Theorem}[section]

\newtheorem{propo}[theo]{Proposition}

\newtheorem{coro}[theo]{Corollary}
\newtheorem{rem}[theo]{Remark}

\newcommand\Surj{\operatorname{\it Surj}}
\newcommand\Inj{\operatorname{\it Inj}}

\newcommand\ob{\operatorname{ob}}

\newcommand\Lan{\operatorname{Lan}}

\newcommand\Alg{\operatorname{\bf Alg}}

\newcommand\Pos{\operatorname{\bf Pos}}

\newcommand\op{\operatorname{op}}

\newcommand\Set{\operatorname{\bf Set}}
\newcommand\Met{\operatorname{\bf Met}}

\newcommand\Str{\operatorname{\bf Str}}

\newcommand\CPO{\operatorname{\bf CPO}}

\newcommand\ca{\mathcal {A}}

\newcommand\cd{\mathcal {D}}

\newcommand\ck{\mathcal {K}}

\newcommand\ct{\mathcal {T}}

\newcommand\cv{\mathcal {V}}

 \newbox\noforkbox \newdimen\forklinewidth
\forklinewidth=0.3pt \setbox0\hbox{$\textstyle\smile$}
\setbox1\hbox to \wd0{\hfil\vrule width \forklinewidth depth-2pt
 height 10pt \hfil}
\wd1=0 cm \setbox\noforkbox\hbox{\lower 2pt\box1\lower
2pt\box0\relax}

\date{November 4, 2024}
 
\begin{document}
\title[Discrete Lawvere theories and monads]
{Discrete Lawvere theories and monads}
\author[J. Rosick\'{y}]
{J. Rosick\'{y}}
\thanks{Supported by the Grant Agency of the Czech Republic under the grant 22-02964S} 
 
\address{\newline J. Rosick\'{y}\newline
Department of Mathematics and Statistics,\newline
Masaryk University, Faculty of Sciences,\newline
Kotl\'{a}\v{r}sk\'{a} 2, 611 37 Brno,\newline
 Czech Republic}
\email{rosicky@math.muni.cz}

\begin{abstract}
We show that, under certain assumptions, strongly finitary enriched monads are given by discrete enriched Lawvere theories. On the other hand, monads given by discrete enriched Lawvere theories preserve surjections.
\end{abstract}
\maketitle
\section{Introduction}
Linton \cite{L} proved that Lawvere theories correspond to finitary monads on the category $\Set$ of sets. Much later, Power \cite{P} introduced enriched Lawvere theories and showed that they correspond to enriched finitary monads on every category $\cv$ which is locally finitely presentable as a symmetric monoidal closed category. His enriched
Lawvere theories have finitely presentable objects of $\cv$ as arities. 
In $\Set$, such objects are discrete in the sense that they are finite
coproducts of the monoidal unit $I$. Discrete Lawvere theories from Power \cite{P1} take discrete finite objects as arities. A natural question is to characterize monads induced by discrete Lawvere theories.
The same question can be asked for every regular cardinal $\kappa$ where
discrete $\kappa$-ary Lawvere theories have $\kappa$-small coproducts of $I$ as arities.

Strongly finitary functors were introduced by Kelly and Lack \cite{KL}
as functors which are the left Kan extension of their restriction on finite discrete objects. They showed that strongly finitary functors are closed under composition in every cartesian closed complete and cocomplete category. Similarly, one says that a functor is strongly $\kappa$-ary if it is the left Kan extension of its restriction on $\kappa$-small discrete objects. Following Bourke and Garner \cite[Theorem 43]{BG}, if strongly $\kappa$-ary functors are closed under
composition then strongly $\kappa$-ary monads are precisely monads given by discrete $\kappa$-ary Lawvere theories. Consequently, strongly finitary monads coincide with monads given by discrete finitary Lawvere theories
in every cartesian closed complete and cocomplete category. More generally, Borceux and Day \cite{BD} proved this for every $\pi$-category.
G. Tendas observed that the argument of \cite{KL} still works when finite products are absolute limits in $\cv$. The reason is that the functor $L$ from \cite{KL} automatically preserves finite products.

We will deal with bases $\cv$ where surjections, i.e., morphisms $f$
such that $\cv(I,f)$ is a surjective, form a left part of an enriched
factorization system on $\cv$. A typical example is the category $\Met$
of metric spaces (with distances $\infty$ allowed) and nonexpanding maps.
Our main result is that strongly $\kappa$-ary monads are given by discrete $\kappa$-ary Lawvere theories provided that $\cv$ is locally $\kappa$-generated and locally $\lambda$-presentable for $\lambda\geq\kappa$.
For instance, strongly finitary monads on $\Met$ are given by discrete finitary Lawvere theories. This was proved by Ad\' amek, Dost\' al and Velebil in \cite{ADV1}. But it is an open problem whether, conversely, monads given by discrete finitary Lawvere theories in $\Met$ are strongly finitary. We will only show that monads induced by discrete Lawvere theories preserve surjections, i.e., that discrete Lawvere theories are discrete equational theories in the sense of \cite{R3}. This question was asked in \cite[4.7]{R3}.


{\bf Acknowledgement.} The author is grateful to Giacomo Tendas for valuable discussions and comments.

\section{Preliminaries}
Let $\cv$ be a symmetric monoidal closed category with the unit object $I$
and the underlying category $\cv_0$. Following \cite{K1}, $\cv$ is locally $\lambda$-presentable as a symmetric monoidal closed category 
if $\cv_0$ is locally $\lambda$-presentable, $I$ is $\lambda$-presentable in $\cv_0$ and $X\otimes Y$ is $\lambda$-presentable in $\cv_0$ whenever $X$ and $Y$ are $\lambda$-presentable in $\cv_0$. This ensures that $\lambda$-presentable objects in the enriched sense and in the ordinary sense coincide. We will denote by $\cv_\lambda$ the (representative) small, full subcategory consisting 
of $\lambda$-presentable objects.

The underlying functor $\cv_0(I,-):\cv_0\to\Set$ has a left adjoint $-\cdot I$ sending a set $X$ to the coproduct $X\cdot I$ of $X$ copies of $I$ in $\cv_0$. Objects $X\cdot I$ will be called \textit{discrete}
(see \cite{R3}). Every object $V$ of $\cv$ determines a discrete object $V_0=\cv_0(I,V)\cdot I$ and morphisms $\delta_V:V_0\to V$ given by the counit of the adjunction. Every morphism $f:V\to W$ determines the morphism $f_0=\cv_0(I,f)\cdot I$ between the underlying discrete objects.
We will denote by $\cd_\lambda$ the (representative) full subcategory consisting of discrete $\lambda$-presentable objects. If $\cv_0(I,-)$
preserves $\lambda$-presentable objects, then $|\cv_0(I,X)|<\lambda$
for every $\lambda$-presentable object $X$. Hence $X_0$ is $\lambda$-presentable as well.

A morphism $f:A\to B$ will be called a \textit{surjection} if $\cv_0(I,f)$ is surjective (see \cite{R3}). For instance, every $\delta_V$ is  surjective. Let $\Surj$ denote the class of all surjections in $\cv_0$ and let $\Inj$ be the class of morphisms of $\cv_0$ having the unique right lifting property w.r.t. every surjection.
Morphisms from $\Inj$ will be called \textit{injections}. $(\Surj,\Inj)$
is a factorization system in $\cv_0$ if and only if $\Surj$ is closed under colimits (see \cite[3.3]{R3}). It is a $\cv$-factorization system
if and only if $\Surj$ is moreover closed under tensors (\cite[5.7]{LW}).
Surjections are epimorphisms provided that $I$ is a generator because then
$\cv(I,-)$ is faithful, thus reflects epimorphisms.

Let $\ca$ be a small, full, dense sub-$\cv$-category of $\cv$ with the inclusion $K:\ca\to\cv$. Objects of $\ca$ are called \textit{arities}. Then an $\ca$-\textit{pretheory} is an identity-on-objects $\cv$-functor $J:\ca\to\ct$. A $\ct$-\textit{algebra} is an object $A$ of $\cv$ together with a $\cv$-functor $\hat{A}:\ct^{\op}\to\cv$ whose composition with $J^{\op}$ is $\cv(K-,A)$. Hence every morphism $f:JY\to JX$ induces an $(X,Y)$-ary operation $\tilde{A}(f):A^X\to A^Y$ on $A$. Every $\ca$-pretheory $\ct$ induces a $\cv$-monad $T:\cv\to\cv$ given by its $\cv$-category $\Alg(\ct)$ of algebras. Conversely, a $\cv$-monad $T$ induces an $\ca$-pretheory $J:\ca\to\ct$ where $\ct$ is the full subcategory of $\Alg(T)$ consisting of free agebras on objects from $\ca$ and $J$ is the domain-codomain restriction of the free algebra functor. An $\ca$-pretheory is an $\ca$-\textit{theory} if it is given by its monad. Then $\hat{A}$ is the hom-functor $\Alg(\ct)(-,A)$ restricted on free algebras over $\ca$. All this is explained in \cite{BG}) where $\ca$-theories are characterized ($\ct$-algebras are called \textit{concrete $\ct$-models}).

Under a $\lambda$-\textit{ary $\cv$-theory} we will mean
a $\cv_\lambda$-theory. Following \cite{BG}, $\lambda$-ary $\cv$-theories correspond to $\cv$-monads on $\cv$ preserving
$\lambda$-filtered colimits. They are called $\lambda$-\textit{ary
$\cv$-monads}. Since $\cd_\lambda$ is dense in $\cv$, we can consider 
 $\cd_\lambda$-theories which correspond to $\lambda$-ary discrete Lawvere theories of \cite{P1}. On $\Met$, finitary discrete Lawvere theories correspond to varieties of unconditional quantitative theories of \cite{MPP1,MPP} (see (\cite[4.7]{R}). On the category $\Pos$  of posets, finitary discrete Lawvere theories correspond to varieties of ordered algebras of \cite{Bl} (see (\cite[4.7]{R}).

We will say that a $\cv$-functor $H:\cv\to\cv$ is \textit{strongly $\lambda$-ary} if it is $\Lan_K HK$ where $K:\cd_\lambda\to\cv$ is the inclusion (see \cite{KL} for $\lambda=\aleph_0)$. A $\cv$-monad $T$ is
strongly $\mu$-ary if $T$ is strongly $\mu$-ary as a $\cv$-functor. Let $\tilde{K}:\cv\to [(\cd_\lambda )^{\op},\cv]$ be the induced fully faithful $\cv$-functor.

\section{Strongly $\lambda$-ary monads}
Recall that an object $A$ in a category $\ck$ is $\mu$-generated if its hom-functor $\ck(A,-):\ck\to\Set$ preserves $\mu$-directed colimits of monomorphisms. A cocomplete category $\ck$ is locally $\mu$-generated  
if it has a set $\ca$ of $\mu$-generated objects such that every object
is a $\mu$-directed colimit of its subobjects from $\ca$ (see \cite{AR3}). Like in the locally presentable case, $\cv$ is locally $\mu$-generated as a symmetric monoidal closed category if $\cv_0$ is locally $\mu$-generated, $I$ is $\mu$-generated in $\cv_0$ and $X\otimes Y$ is $\mu$-generated in $\cv_0$ whenever $X$ and $Y$ are $\mu$-generated in $\cv_0$. This ensures that $\mu$-generated objects in the enriched sense and in the ordinary sense coincide. 

\begin{theo}\label{discrete}
Let $\cv$ be locally $\lambda$-presentable and locally $\mu$-generated as a symmetric monoidal closed category for $\mu\leq\lambda$. Assume that $I$ is a generator, $\cv_0(I,-)$ preserves $\lambda$-presentable objects and that $(\Surj,\Inj)$ is a $\cv$-factorization system. Then strongly $\mu$-ary $\cv$-monads on $\cv$ are given by $\mu$-ary discrete  Lawvere theories.
\end{theo}
\begin{proof}
Let $T$ be a strongly $\mu$-ary $\cv$-monad on $\cv$. Since $\cd_\mu$ consists of $\lambda$-presentable objects, $T$ preserves $\lambda$-directed colimits (see \cite[5.29]{K}). Following 
\cite{BG}, $\cv^T$ is equivalent to $\Alg(\ct)$ where $\ct$ is an $\lambda$-ary $\cv$-theory whose $(X,Y)$-ary operations correspond to morphisms $f:FY\to FX$; here $U:\Alg(T)\to\cv$ is the forgetful functor and $F$ is its enriched left adjoint. 

We will show that $T$ preserves surjections. Let $f:X\to Y$
be a surjection. Express $Y$ and $X$ as weighted colimits $\tilde{K}X\ast K$ and $\tilde{K}Y\ast K$ of $\mu$-presentable discrete objects. Following \cite[5.29]{K}, $T$ preserves these weighted colimits, i.e., 
$TX=\tilde{K}X\ast TK$ and $TY=\tilde{K}Y\ast TK$. Hence  
$$
T(f)=\tilde{K}(f)\ast TK.
$$ 
Since,
$$
\tilde{K}(f)_{Z\cdot I}= [Z\cdot I,f]= f^Z
$$
and surjections are closed under products (\cite[3.2(2)]{R3}), 
$\tilde{K}(f):\tilde{K}X\to \tilde{K}Y$ is pointwise surjective. The dual of \cite[4.5]{LW} implies that $T(f)$ is surjective. 

Let $\ct_d$ be the restriction of $\ct$ on $\cd_\lambda$ and
$$
R: \Alg(\ct)\to\Alg(\ct_d)
$$ 
the reduct functor. Clearly, it is a $\cv$-functor and we will prove that it is an isomorphism. At first, we will show that $R$ is an embedding. It is clearly faithful. Following \cite[4.4]{R3}, for every  $f:FY\to FX$, there is $f_0:FY_0\to FX_0$ such that $F(\delta_X)f_0=fF(\delta_Y)$. Let $\omega$ be an $(X,Y)$-ary operation corresponding to $f:FY\to FX$ and $\omega_0$ an $(X_0,Y_0)$-ary operation corresponding to $f_0$. Consider a $\ct$-algebra $A$. We have
$$
(UA)^{\delta_Y}\omega_A=(\omega_0)_A(UA)^{\delta_X}
$$
Since $I$ is a generator, $\delta_Y$ is an epimorphism and thus $(UA)^{\delta_Y}$ is a monomorphism. Hence $(\omega_0)_A=(\omega_0)_B$ implies that $\omega_A=\omega_B$ for any $\ct$-algebras $A$ and $B$. Thus $R$ is an embedding.

We will show that $R$ is surjective on objects. Let $A$ be a $\ct_d$-algebra. Again, let $\omega$ be an $(X,Y)$-ary operation corresponding to $f:FY\to FX$ and $\omega_0$ an $(X_0,Y_0)$-ary operation corresponding to $f_0$. Since $\cv_0(I,-)$ preserves $\lambda$-presentable objects, $X_0$ and $Y_0$ are $\lambda$-presentable. Since $A$ is an $\ct_d$-algebra, we get 
$$
(\omega_0)_A=\tilde{A}(f_0):(UA)^{X_0}\to (UA)^{Y_0}.
$$ 
We will show that there is a unique $\tilde{f}:(UA)^{X_0}\to (UA)^Y$ such that $(UA)^{\delta_Y}\tilde{f}=(\omega_0)_A$. Consider $\varphi:\tilde{K}Y\to [(UA)^{X_0},(UA)^K]$ sending 
$a:Z\cdot I\to Y$ to
$$
(UA)^{X_0} \xrightarrow{\ (\omega_0)_A } (UA)^{Y_0} \xrightarrow{\ (UA)^{a_0}} (UA)^{Z\cdot I}
$$
where $a_0:Z\cdot I\to Y_0$ is determined by $a$, i.e., $\delta_Y a_0=a$.
Since $(UA)^Y$ is a weighted limit $\{\tilde{K}Y,(UA)^K\}$, we get $\tilde{f}:(UA)^{X_0}\to (UA)^Y$ induced by $\varphi$. Then
$(UA)^{\delta_Y}\tilde{f}=\{\tilde{K}\delta_Y,(UA)^K\}\tilde{f}$ is induced by $\varphi\tilde{K}\delta_Y$. Put $\omega_A=\tilde{f}(UA)^{\delta_X}$. We have
$$
(UA)^{\delta_Y}\omega_A=(UA)^{\delta_Y}\tilde{f}(UA)^{\delta_X}=(\omega_0)_A(UA)^{\delta_X}.
$$

We have to how that this makes $A$ an $\ct$-algebra. Let $\rho$ be a $(Y,Z)$-ary operation corresponding to $g:FY\to FY$. We have to show that
$(\rho\omega)_A=\rho_A\omega_A$. Put $h=gf$. Then
$$
\rho_A\omega_A=\tilde{g}(UA)^{\delta_Y}\tilde{f}(UA)^{\delta_X}=
\tilde{g}(\omega_0)_A(UA)^{\delta_X}.
$$
Since
$$
(UA)^{\delta_Z}\tilde{g}(\omega_0)_A=(\rho_0)_A(\omega_0)_A=((\rho\omega)_0)_A=(UA)^{\delta_Z}\tilde{h},
$$
the unicity implies that $\tilde{h}=\tilde{g}(\omega_0)_A$. Hence
$$
(\rho\omega)_A=\tilde{h}(UA)^{\delta_X}=\tilde{g}(\omega_0)_A(UA)^{\delta_X}=\tilde{g}(UA)^{\delta_Y}\omega_A=\rho_A\omega_A.
$$.

This proves that the reduct functor $R$ is surjective on objects. It 
remains to show that $R$ is full. Consider $\ct$-algebras $A$ and $B$ 
and a homomorphism $h:RA\to RB$. Then
$$
(UB)^{\delta_Y}h^Y\omega_A=h^{Y_0}(UA)^{\delta_Y}\omega_A=h^{Y_0}(\omega_0)_A(UA)^{\delta_X}=(\omega_0)_Bh^{X_0}(UA)^{\delta_X}
$$
and
$$ 
(UB)^{\delta_Y}\omega_Bh^X=(\omega_0)_B(UB)^{\delta_X}h^X=(\omega_0)_Bh^{X_0}(UA)^{\delta_X}
$$
Since $(UB)^{\delta_Y}$ is a monomorphism, we get
$$
h^Y\omega_A=\omega_Bh^X.
$$
Thus $h:A\to B$ is a homomorphism. We have proved that $R$ is full.

Since $\cd_\mu$ consists of $\mu$-generated objects, $T$ preserves $\mu$-directed colimits of mono\-morphisms (more precisely, it sends $\mu$-directed colimits of monomorphisms to $\mu$-directed colimits, see \cite[5.29]{K}). Let $\ct'_d$ is a subtheory of $\ct_d$ consisting of $(X,Y)$-ary operations where $X$ is discrete $\lambda$-presentable and $Y$ discrete $\mu$-presentable. \cite{R} 3.17 implies that the reduct 
$\cv$-functor 
$$
\Alg(\ct_d)\to\Alg(\ct'_d)
$$ 
is an equivalence. 

Let $\ct_{d\mu}$ be the restriction of $\ct_d$ on $\cd_\mu$. Every discrete $\lambda$-presentable object $X$ is a $\mu$-directed colimit $x_i:X_i\to X$ of discrete $\mu$-presentable objects $X_i$ and split monomorphisms. Consider an $(X,Y)$-ary operation of $\ct'_d$ corresponding to a morphism $f:FY\to FX$, hence to its adjoint transpose 
 $f^\ast:Y\to TX$. Since $Tx_i:TX_i\to TX$ is a $\mu$-directed colimit of $TX_i$ and split monomorphisms, $f^\ast$ factorizes through some
$Tx_i$ as $f^\ast=(Tx_i)g$. Then $g:Y\to TX_i$ yields an $(X_i,Y)$-ary
operation $g^\ast$. Consequently, the reduct $\cv$-functor
$$
\ct'_d\to\ct_{d\mu}
$$
is an equivalence. Hence the reduct functor $\Alg(\ct)\to\Alg(\ct_{d\mu})$ is a $\cv$-equivalence. Thus $T$ is given by a $\mu$-ary discrete Lawvere theory $\ct_{d\mu}$.
\end{proof}

\begin{rem}
{
\em
(1) $\Pos$ satisfies the assumptions of \ref{discrete} for $\lambda=\mu=\aleph_0$ while $\Met$ for $\lambda=\aleph_1$ and $\mu=\aleph_0$.
 
(2) In the terminology of \cite{BG}, we have shown that every $\cd_\mu$-induced theory is $\cd_\mu$-nervous (under the assumptions 
of \ref{discrete}). \cite[Theorem 43]{BG} shows that, if $\cd_\mu$
is saturated in the sense that strongly $\mu$-ary functors $H:\cv\to\cv$
are closed under composition, then $\cd_\mu$-induced and $\cd_\mu$-nervous
monads coincide. This means that strongly $\mu$-ary monads and monads given by $\mu$-ary discrete Lawvere theories coincide. Following \cite{KL}, this is true in every cartesian closed $\cv$. This covers both $\Pos$ and the category $\omega$-$\CPO$ of posets with joins of non-empty $\omega$-chains (strongly finitary monads in these categories were characterized in \cite{ADV} and \cite{ADV2}). As we have mentioned in the introduction, this is also true when finite products are absolute which include abelian categories $\cv$.
 
}
\end{rem}

\section{Discrete Lawvere theories} 

Following \cite{BG}, a \textit{signature} is a functor $\ob\cd_\mu\to\cv$ where
$\ob\cd_\mu$ is the discrete $\cv$-category on the set of objects of $\cd_\mu$. Thus a signature $\Sigma$ specifies for each discrete $\mu$-presentable object $X\cdot I$ an object $\Sigma X$ of operations of input arity $X$, i.e., of $(X,I)$-ary operations. A signature $\Sigma$ will be called \textit{discrete} if all objects $\Sigma X$ are discrete.

\begin{propo}\label{free}
A free $\cv$-monad over a discrete signature is strongly $\mu$-ary.
\end{propo}
\begin{proof}
Following the proof of \cite[Proposition 55]{BG}, a free $\cv$-monad $T_\Sigma$ over $\Sigma$ is is a free monad over the functor
$$
\coprod\limits_X (-^X)\cdot\Sigma X
$$
where $|X|<\mu$. Since $\Sigma X$ is discrete, $T_\Sigma$ is a coproduct of powers $-^X$, $|X|<\mu$. Since strongly $\mu$-ary functors
are closed under coproducts (\cite[Lemma 57]{BG}), it suffices to show that powers $-^X$ , $|X|<\mu$ are strongly $\mu$-ary. But this follows from the enriched Yoneda lemma because, for a $\cv$-functor $S:\cv\to\cv$, we
have
$$
[(\cd_\mu)^{\op},\cv](\tilde{K}X,SK)\cong SX\cong [\cv,\cv](\cv(KX,-),S).
$$
\end{proof}

\begin{theo}
Monads induced by $\mu$-ary discrete Lawvere theories preserve surjections.
\end{theo}
\begin{proof}
Let $\ct$ be a $\mu$-ary discrete Lawvere theory. Let $\mathbb F$ be the induced enriched functional language with discrete arities and $\Str(\mathbb F)$ its $\cv$-category of $\mathbb F$-structures in the sense of \cite[3.1 and 3.3]{RTe}. Then $\Alg(\ct)$ is a Birkhoff subcategory of $\Str(\mathbb F)$, hence a reflective subcategory
of $\Str(\mathbb F)$ whose reflections $\rho_A$ are strong epimorphism
(see \cite[5.2]{R3}). Following \cite[3.5]{R3}, $\rho_A$ are surjections.
Since the monad $T_\mathbb F$ induced by $\Str(\mathbb F)$  is strongly $\mu$-ary (\ref{free}), it preserves surjections (following the proof of \ref{discrete}). Since $\rho:T_\mathbb F\to T$ is pointwise surjective, $T$ preserves surjection as well.
\end{proof}

\begin{coro}
$\mu$-ary discrete Lawvere theories are discrete equational theories in the sense of \cite{R3}.
\end{coro}


\begin{thebibliography}{EAPT}
  \itemsep=2pt

 \bibitem{ADV} J. Ad\' amek, M. Dost\' al and J. Velebil, {\em A categorical view of varieties of ordered algebras}, Math. Struct. Comp. Sci. 32 (2022), 349-373.
 
\bibitem{ADV1} J. Ad\' amek, M. Dost\' al and J. Velebil, {\em Quantitative algebras and a classification of metric monads}, arXiv:2210.01565.

\bibitem{ADV2} J. Ad\' amek, M. Dost\' al and J. Velebil, {\em Sifted colimits, strongly finitary monads and continuous algebras}, arXiv:2210.2301.05730.


\bibitem{AR3} J. Ad\'{a}mek and J. Rosick\'{y}, {\em Locally Presentable and Accessible Categories}, Cambridge University Press 1994.
  




\bibitem{Bl} S. Bloom, {\em Varieties of ordered algebras}, J. Comput. System Sci. 13 
(1976), 200-212.

\bibitem{BD} F. Borceux and B. Day, {\em Universal algebra in a closed category}, J. Pure Appl. Algebra 16 (1980), 133-147.

\bibitem{BG} J. Bourke and R. Garner, {\em Monads and theories}, Adv. Math. 351 (2019), 1024-1071.








\bibitem{K} G. M. Kelly, {\em Basic Concepts of Enriched Category Theory}, Cambridge Univ. Press 1982.

\bibitem{K1}
G.~M. Kelly, {\em Structures defined by finite limits in the enriched context, I}, Cah. Top. G\'{e}om. Diff. Cat., 23 (1982), 3-42.

\bibitem{KL} G. M. Kelly and S. Lack,
{\em Finite-product preserving functors, Kan extensions and strongly finitary 2-monads}, App. Categ. Struct. 1 (1993), 85-94.




\bibitem{L} F. E. J. Linton, {\em Some aspects of equational categories}, 
In: Conf. Categ. Algebra (La Jolla 1965), Springer 1966, 84-94.

\bibitem{LW} R. B. B. Lucyshyn-Wright, {\em Enriched factorization systems}, Th. Appl. Cat. 29 (2014), 475-495.



\bibitem{MPP1} R. Mardare, P. Panangaden and G. Plotkin, {\em Quantitative algebraic reasoning}, In: Proc. LICS 2016, 700-709.
 
\bibitem{MPP} R. Mardare, P. Panangaden and G. Plotkin, {\em On the axiomatizability of quantitative algebras}, In: Proc. LICS 2017.

\bibitem{P} J. Power, {\em Enriched Lawvere theories}, Theory Appl. Categ 6 (1999), 83-93.

\bibitem{P1} J. Power, {\em Discrete Lawvere theories}, In: Algebra and coalgebra in computer science, Lect. Not. Comp. Sci. 3629, Springer 2005, 348-363.


\bibitem{R} J. Rosick\' y, {\em Metric monads}, Math. Struct. Comp. Sci. 31 (2021), 535-552.

\bibitem{R3} J. Rosick\' y, {\em Discrete equational theories}, Math. Struct. Comp. Sci. 34 (2024), 147-160.
 
\bibitem{RTe} J. Rosick\' y and G. Tendas, {\em Towards enriched universal algebra}, arXiv:2310.11972.
  
   
 
 
\end{thebibliography}
\end{document}